\documentclass[11pt]{article}

\usepackage{fullpage}

\usepackage{amsmath,amsfonts,amssymb,graphics,amsthm}
\usepackage[dvips]{graphicx}
\usepackage{subfigure}
\usepackage{nicefrac}
\usepackage[pdfborder={0 0 0},colorlinks=true,linkcolor=black,citecolor=black,urlcolor=blue]{hyperref}
\usepackage{caption}
\allowdisplaybreaks

\usepackage{xcolor}

\usepackage{hyperref}
\hypersetup{
    colorlinks=true,
    linkcolor=blue,
    citecolor=red,
    urlcolor=blue,
    pdfborder={0 0 0}
}

\usepackage{cleveref}

\usepackage{cleveref}
  \crefname{theorem}{Theorem}{Theorems}
  \crefname{thm}{Theorem}{Theorems}
  \crefname{lemma}{Lemma}{Lemmas}
  \crefname{lem}{Lemma}{Lemmas}
  \crefname{remark}{Remark}{Remarks}
    \crefname{rmk}{Remark}{Remarks}
  \crefname{prop}{Proposition}{Propositions}
\crefname{notation}{Notation}{Notations}
\crefname{claim}{Claim}{Claims}
  \crefname{defn}{Definition}{Definitions}
   \crefname{definition}{Definition}{Definitions}
  \crefname{cor}{Corollary}{Corollaries}
  \crefname{section}{Chapter}{Chapters}
  \crefname{subsection}{Section}{sections}
  \crefname{figure}{Figure}{Figures}
    \crefname{assumption}{Assumption}{Assumptions}

\newtheorem{theorem}{Theorem}[section]

\newtheorem{lemma}[theorem]{Lemma}

\newtheorem{prop}[theorem]{Proposition}

\numberwithin{equation}{section}

\theoremstyle{definition}

\usepackage[T1]{fontenc}
\usepackage[utf8]{inputenc}

\def\cR{\mathcal{R}}

\def\cK{\mathcal{K}}

\def\P{\mathbb{P}}

\def\H{\mathbb{H}}
\def\D{\mathbb{D}}

\def\E{\mathbb{E}}
\def\C{\mathbb{C}}
\def\R{\mathbb{R}}

\def\eps{\varepsilon}

\DeclareMathOperator{\var}{Var}

\DeclareMathOperator{\hcap}{hcap}

\DeclareMathOperator{\diam}{Diam}

\renewcommand{\l}{\lambda}
\renewcommand{\d}{\delta}

\newcommand{\old}[1]{{}}

\newcommand{\blu}[1]{{\textcolor{black}{#1}}}

\usepackage{comment}

\usepackage{makeidx}
\makeindex

\begin{document}

\title{Explosive growth for a constrained Hastings--Levitov \\
aggregation model}


 \author{
\begin{tabular}{c} Nathana\"el Berestycki\\
\small University of Vienna \\
 \small nathanael.berestycki@univie.ac.at
 \end{tabular}
\begin{tabular}{c} Vittoria Silvestri\\\small University of Rome La Sapienza \\
\small silvestri@mat.uniroma1.it
\end{tabular}
}

\date{}

\maketitle

\abstract{We consider a constrained version of the HL$(0)$ Hastings--Levitov model of aggregation in the complex plane, in which particles can only attach to the part of the cluster that has already been grown. Although one might expect that this gives rise to a non-trivial limiting shape, we prove that the cluster grows explosively: in the upper half plane, the aggregate accumulates infinite diameter
as soon as it reaches positive capacity. More precisely, we show that after $nt$ particles of (half-plane) capacity $1/(2n)$ have attached, the diameter of the shape is highly concentrated around $\sqrt{t\log n}$, uniformly in $t\in [0,T]$. This illustrates a new instability phenomenon for the growth of single trees/fjords in unconstrained HL$(0)$.}

\section{Introduction}
We study the growth of an aggregate on the upper half-plane, described by the composition of random conformal maps, which is a modification of the well known Hastings--Levitov HL$(0)$ aggregation model. Let us succinctly describe it. Each map represents the arrival of a new particle, and the randomness comes from the attachment locations on the boundary of the aggregate.

Denote by $\mathbb{H} =\{ z \in \mathbb{C} : \Im (z) > 0\}$ the upper half-plane.
For $n \geq 1$ let $F: \mathbb{H} \to \mathbb{C} $ be the slit map
	\begin{equation}\label{eq:F}
	F(z) = \sqrt{z^2 - 1/n} , 
	\end{equation}
which grows a vertical slit of length $1/\sqrt{n} $ and half--plane capacity $1/(2n)$ at $0$ (see Figure \ref{F:slitmap} below). 

\begin{figure}[!h]
  \begin{center}
    \includegraphics[width=.6\textwidth]{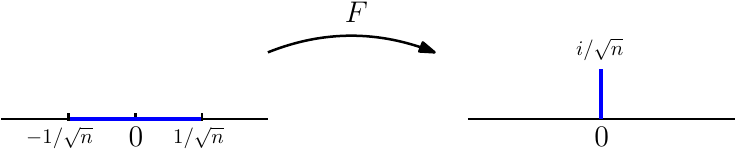}
    \caption{The map $F$ growing a slit of height $1/\sqrt{n}$ at the origin.
    \label{F:slitmap}}
  \end{center}
\end{figure}


Then, if $x \in \mathbb{R}$,  the map
	\[ F_x(z) =  x+ \sqrt{(z-x)^2 -1/n }\]
grows a vertical slit of height $1/\sqrt{n}$ at $x$. We use this to define a growth model as follows. Let $\Phi_0 (z) = z$, \blu{$\Phi_1 (z) = F (z)  $ with $F$ as in \eqref{eq:F} above,} and denote\footnote{For a set $\mathcal{K} \subseteq \mathbb{H}$ and a map $f: \mathbb{H} \to\mathbb{C}$, we write $f(\mathcal{K}) =\{ f(z) : z \in \mathcal{K}\}$.}  by $\mathcal{K}_1 = \overline{\mathbb{H} \setminus \Phi_1 (\mathbb{H} ) }$ the closure of the complement of $\Phi_1(\mathbb{H} )$ in $\mathbb{H}$, so that $\mathcal{K}_1$ is a compact set consisting of a vertical slit at $0$.
Denote by  $\Gamma_1(z) = \Phi_1^{-1}(z)$ the associated mapping--out function. Then $\Gamma_1$ maps the slit $\mathcal{K}_1$ to the interval $[-L_1 , R_1 ]$ on the real line, with $L_1=R_1=1/\sqrt{n}$. To the initial cluster $\mathcal{K}_1$ we attach particles as follows.
Suppose inductively that for $k\geq 1$ we have defined the map $\Phi_k $ that grows a random compact set $\mathcal{K}_k$, which we think of as the cluster up to the $k^{th}$ arrival. Let $\Gamma_k = \Phi_k^{-1} $ denote the associated mapping out function, and assume that  $\Gamma_k (\mathcal{K}_k ) $ is an interval on the real line, call it $ [-L_k , R_k ]$. Then choose a uniformly random point $x_{k+1} $ on $[-L_k , R_k ]$, and set
	\begin{equation}\label{Phik}
\Phi_{k+1} = \Phi_k \circ F_{x_k} ,
	\qquad \mathcal{K}_{k+1} = \overline{\mathbb{H} \setminus \Phi_{k+1} (\mathbb{H} )}. \end{equation}
Note that $\mathcal{K}_{k+1} = \mathcal{K}_k \cup P_{k+1} $ where
$P_{k+1} = \overline{ \Phi_k ( [x_{k+1} , x_{k+1} + i /\sqrt{n} ] ) }$ represents the $(k+1)^{th} $ particle.
Unfolding the recursion, we see that this growth mechanism corresponds to setting
	\[ \Phi_k (z) =  F_{x_1} \circ F_{x_2} \circ \cdots \circ F_{x_k} (z) , \]
for $x_1 , x_2 \ldots x_k $ random points with $x_1 =0$ and $x_{j+1} \sim Uniform [-L_j , R_j ]$ for $j\geq 1$. We refer to $x_k$ as the attachment location of the $k^{th}$ particle $P_k$, \blu{since it determines the location $\Phi_{k-1} ( x_k ) $ at which the $k^{th}$ particles attaches to the cluster $\cK_{k-1}$ (see Figure \ref{F:smallcluster} below).}
Note that $L_k , R_k \geq 0$ for all $k\geq 1$ by construction.\\

\begin{figure}[!h]
  \begin{center}
    \includegraphics[width=.75\textwidth]{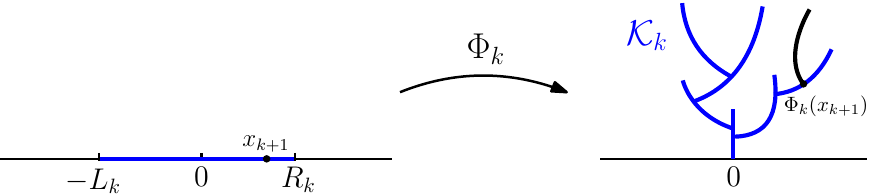}
    \caption{The map $\Phi_k$ maps the interval $[-L_k , R_k]$ to the cluster $\cK_k$ with $k$ particles. \label{F:smallcluster}}
  \end{center}
\end{figure}

The above model is a variant of the well-known Hastings--Levitov model HL($0$), where the growth usually takes place outside the unit disc $\D = \{ |z|\leq 1\}$, rather than on the upper half--plane $\H$ \cite{carleson2001aggregation,hastings1998laplacian,norris2012hastings,silvestri2017fluctuation} (see however \cite{berger2020growth} by Berger, Procaccia and Turner, where this model is introduced in the upper half-plane). 

\begin{figure}[t]
\begin{center}
\includegraphics[width=.4\textwidth]{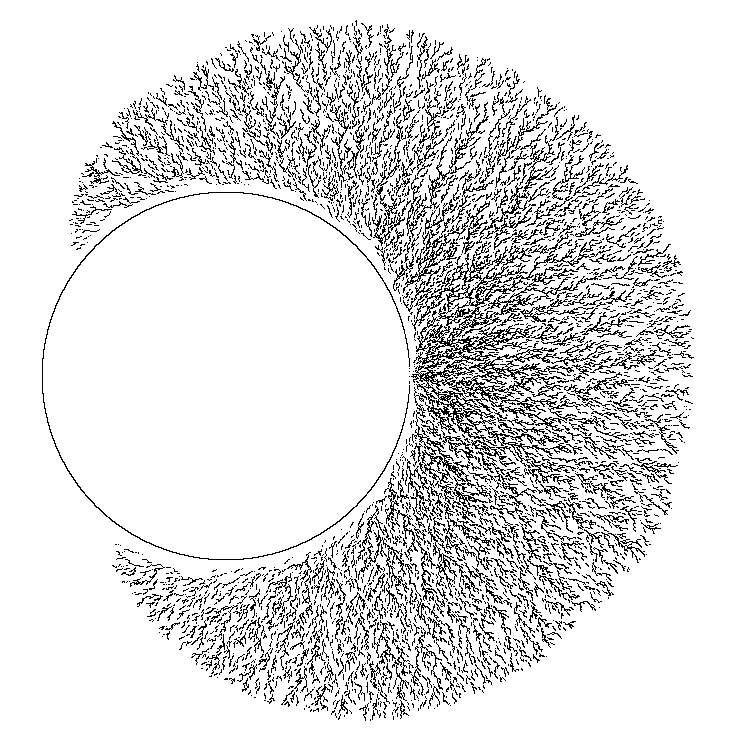}
\end{center}
\caption{Restricted HL$(0)$ on the outside of the unit disc. To (approximately) quote Magritte, this is not a shape theorem.}
\label{F:}
\end{figure}

We call this new model \textbf{constrained Hastings--Levitov}, as the particles are only allowed to attach on the boundary of the grown region, rather than everywhere on the real line. In other words, the growth of the cluster is restricted to a single tree. See Figure \ref{F:} for a simulation of constrained HL$(0)$ outside the unit disc.

\subsection{Main results}


The question is to describe the shape of the cluster $\mathcal{K}_n$ for large $n$. This is implicitly obtained by solving Loewner equation
	\begin{equation}\label{EQn}
	 \begin{cases}
	\dot f_t^{(n)} (z) =  \partial_z f_t ^{(n)} (z)  \displaystyle \int_\mathbb{R} \frac 1{x-z } \mu_t^{(n)} (dx)
\\ f_0^{(n)} (z) = z
\end{cases}
\end{equation}
 driven by the random measure
	\begin{equation}\label{E:drive}
	 \mu^{(n)}_t (dx) = \blu{\sum_{k= 0 }^{\infty}} 1 \Big( t \in \Big[ \frac kn , \frac {k+1}n \Big) \Big)
	\d_{x_k } (dx) ,
	\end{equation}
where $\d_x$ denotes the Dirac's delta measure centered at $x\in \mathbb{R}$,  and $x_1 , x_2 \ldots $ denote the attachment locations of subsequent particles on the real line. Note that, by definition of the model, $x_k $ is uniformly distributed in the interval $[-L_{k-1} , R_{k-1} ]$, so we start by describing the asymptotic behaviour of this random interval for large $n$.

To state the main result, write $\l_k = L_k + R_k$ for the length of the interval after $k$ \blu{arrivals}. As the half-plane capacity of each particle is $1/(2n)$ and half-plane capacity is additive, we obtain a half-plane capacity of order $O(1)$ when $k = tn$ particles are added to the aggregate. Based on this simple observation we typically expect a nontrivial shape in this regime; this is in particular what happens for unconstrained HL$(0)$, and seems consistent with the simulations in Figure \ref{F:}. However this turns out to not be the case: indeed our results show that after $nt$ particles the diameter diverges like $\sqrt{t \log n}$. To prove this, the key is to understand the size of the interval in which particles are allowed to attach after we map out the cluster. This is specified by the result below:

\begin{theorem}\label{T:main}
For any $\eps >0$ and $T>0$ it holds \blu{that}
	\[ \P \left( \sup_{t \in [(\log n)^9/n , T] } \Big| \frac{R_{nt}}{\sqrt{t \log n } } -\frac 12\Big| \leq  \eps \right)
	\geq 1-  e^{-c(\log n )^2} \]
for $n$ large enough ($n\geq n_0(\eps , T)$). Here $c = c(\eps , T)$ is a constant which may depend on $\eps $ and $T$ but not on $n$. Moreover, the same holds with $L_{nt}$ in place of $R_{nt}$. In particular
	\begin{equation}\label{Lmain}
	 \P \left(\sup_{t\in [(\log n)^9 /n , T ]} \Big| \frac{\l_{nt}}{\sqrt{t \log n } } -1\Big| \leq  \eps \right)
	\geq 1-  2e^{-c(\log n )^2}
	\end{equation}
for $n$ large enough.
\end{theorem}
Note in particular that after $nt$ arrivals, even though the capacity of the cluster is of unit order, the length of the interval tends to infinity!


\subsection{Geometric consequences and conjectures}

We use Theorem \ref{T:main} to deduce geometric information \blu{about} the cluster $\mathcal{K}_{nt}$, showing that it has diverging diameter in $n$.
\begin{theorem}\label{T:geometric}
Let $\diam (\cK_{nt}) = \sup\{|x - y|: x, y \in \cK_{nt} \}$ denote the diameter of $\cK_{nt}$. Then
$$
\frac{\diam (\cK_{nt})}{\sqrt{t \log n}} \to 1
$$
in probability. Furthermore,
$$
\max_{z \in \cK_{nt}} \{\Im (z)\} \le  \sqrt{2t},
$$
almost surely.
\end{theorem}

The above theorem says that the shape of the cluster is very elongated, being of order $O(1)$ in height and of order approximately $\sqrt{t \log n}$ in width.

It is tempting to conjecture that the cluster shape (viewed from the outside) is more precisely described by the following Loewner--Kufarev evolution: namely, denote by $f_t(z)$ the deterministic map solution to the Loewner equation
	\begin{equation}\label{EQ}
	 \begin{cases}
	\dot f_t (z) = \partial_z f_t (z)
	\displaystyle \int_\mathbb{R} \frac 1{x-z } \mu_t (dx)
	\\
	f_0(z)=z ,
	\end{cases}
	\end{equation}
where $\mu_t $ denotes the uniform measure on the interval $\big[ - \frac{\sqrt{t \log n} }{2} , \frac{\sqrt{t \log n} }{2} \big]$. Recall the conformal map $\Phi_k$, defined in \eqref{Phik},
which maps the upper half plane $\H$ to $H_k = \H \setminus \cK_k$ and grows the restricted Hastings--Levitov cluster $\cK_k$ up to the $k^{th}$ particle.

Since $\Phi_k$ can be described by a Loewner-Kufarev evolution driven by the measures $\mu_t^{(n)}$, which is close in a suitable sense to the deterministic measure $\mu_t$ defined above, it is natural to guess that $\Phi_{nt}$ is close to the conformal map $f_t$ (indeed we note that a related continuity statement of the Loewner--Kufarev equation was obtained by \cite{bauer2003discrete}, \cite{johansson2009rescaled} and \cite{viklund2012scaling}). Indeed using such an approach (and a stronger uniform continuity of the Loewner--Kufarev evolution), one can prove the following statement:
for any $\eps , t>0$ and compact set $K \subset \mathbb{H}$ there exists a constant $C(t,K)$ such that
	\[ \P \left( \sup_{z \in K } \big| \Phi_{nt}(z) - f_t (z) \big| \leq C(t, K) \, \eps \mbox{ eventually in } n \right) =1. \]
However, the conformal map $f_t$, although deterministic, still depends on $n$ and is in the pointwise limit as $n \to \infty$ close to the trivial identity map, so this statement in itself doesn't convey useful information about the geometry of the cluster.

In fact, formulating a precise statement on how close $\Phi_{nt}$ needs to be to $f_t$ in order to capture something meaningful is in itself not entirely trivial. At the very least one would need to prove that the size of the difference between $\Phi_{nt}$ and $f_t$ is smaller than the size of the discrepancy between $f_t$ and the identity map, and so would need to be formulated quantitatively.

  Let us expand a little on the orders of magnitude that we believe will be involved. We believe that these can be deduced from Theorem \ref{T:geometric}: indeed, since the half-plane capacity of $\cK_{nt}$ is $t$ and the diameter is approximately $\sqrt{t \log n}$, we believe that the typical height of the cluster is of order $O( 1/ \sqrt{ \log n})$ for a fixed time $t$. Therefore we conjecture that  $|f_t(z)- z|$ is of this order of magnitude too.

\paragraph{Simulations.} To complement this discussion we include some simulations of our restricted HL$(0)$ aggregation model on the upper half plane. These simulations do not show the entire cluster but only an envelope of the cluster: roughly speaking they show the conformal image of $\{ z \in \H: \Im (z) = \eps\}$ under the conformal map $\Phi_{nt}$ sending $\H$ to $\H \setminus \cK_{nt}$, for some small $\eps>0$, and for various values of $n$ and $t$; in fact they are parametrised instead by $N$ (the number of particles) and the individual size (i.e., length) which is called $d$ in these pictures (or equivalently, half-plane capacity $d^2/2$). The relation between these parameters is as follows:
\begin{align*}
d^2 &= 1/n,\\
N &= t/d^2 ,
\end{align*}
where $N$ is the number of particles.

\begin{figure}\begin{center}
\includegraphics[width=.8\textwidth]{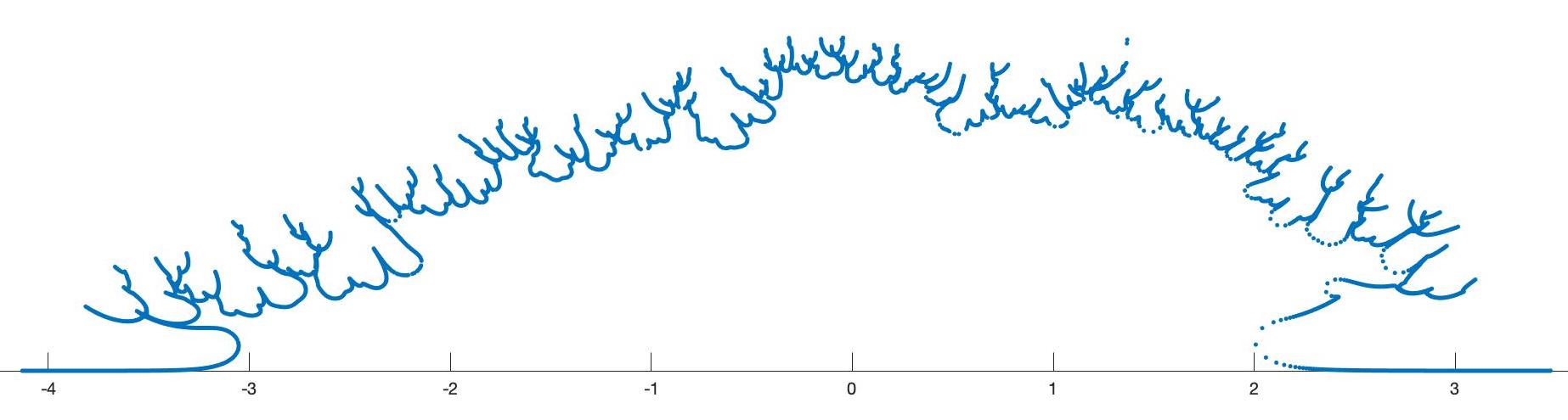}
\vspace{5mm}

\includegraphics[width=.8\textwidth]{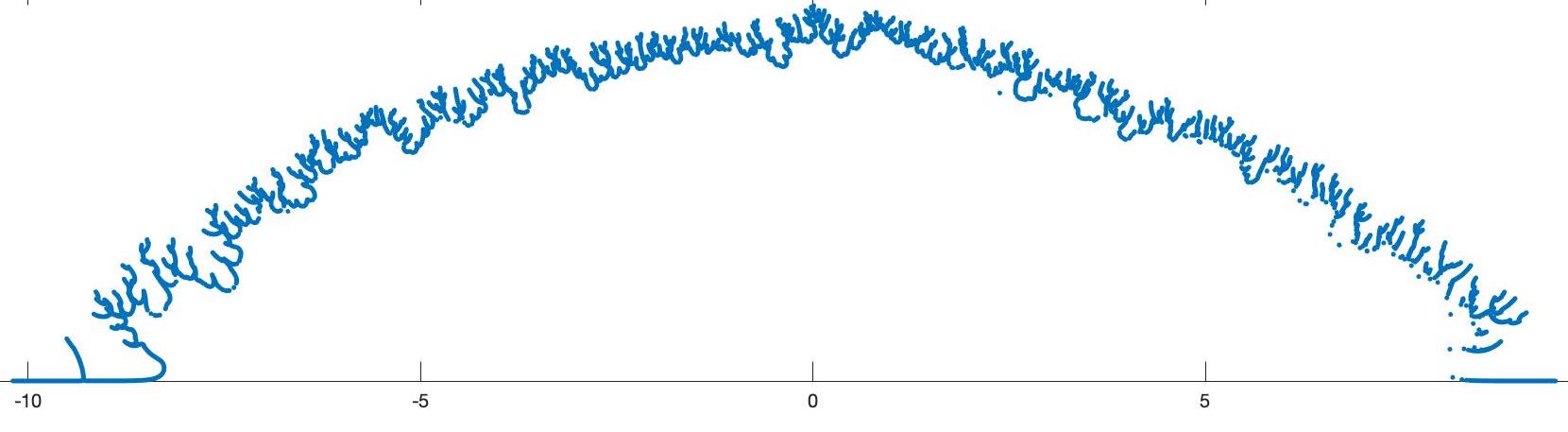}
\vspace{5mm}

\includegraphics[width=.8\textwidth]{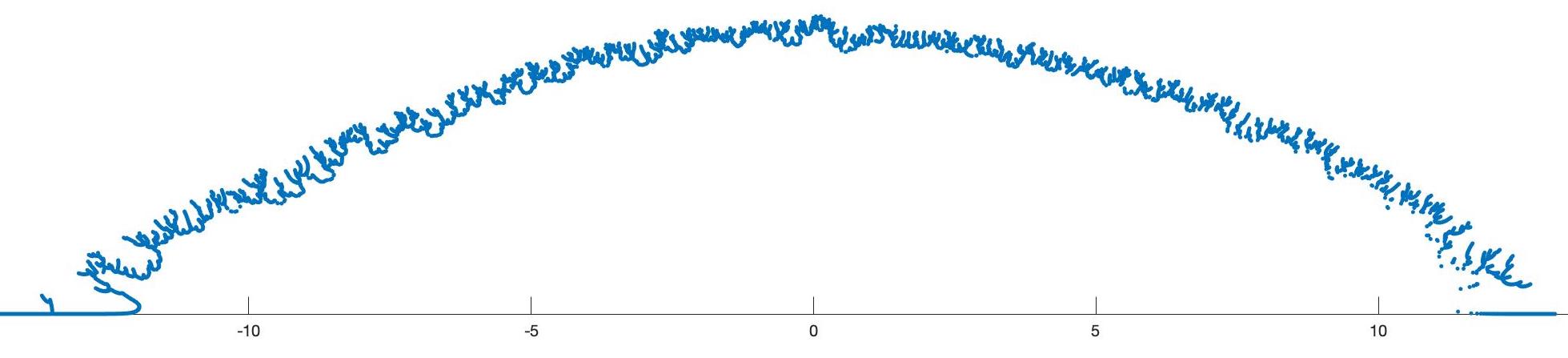}

\end{center}
\caption{Simulations of cluster envelope at $\eps = 10^{-5}$ with $N = 1000, \, 5000$ and $10000$ with $d = 0.1$. This corresponds to $n = 10^{2}$ and $t = 10, \, 50$ and $100$ respectively.
\label{F:clusters}}
\end{figure}

We also include for these same simulations the image of the cluster under the conformal map $\Phi_{k}^{-1}$, which is the interval $[-L_k, R_k]$ of the real line, described by Theorem \ref{T:main} and crucial to the arguments of this paper. The simulations also show the actual locations of where particle successively attach.

\begin{figure}\begin{center}
\includegraphics[width=.48\textwidth]{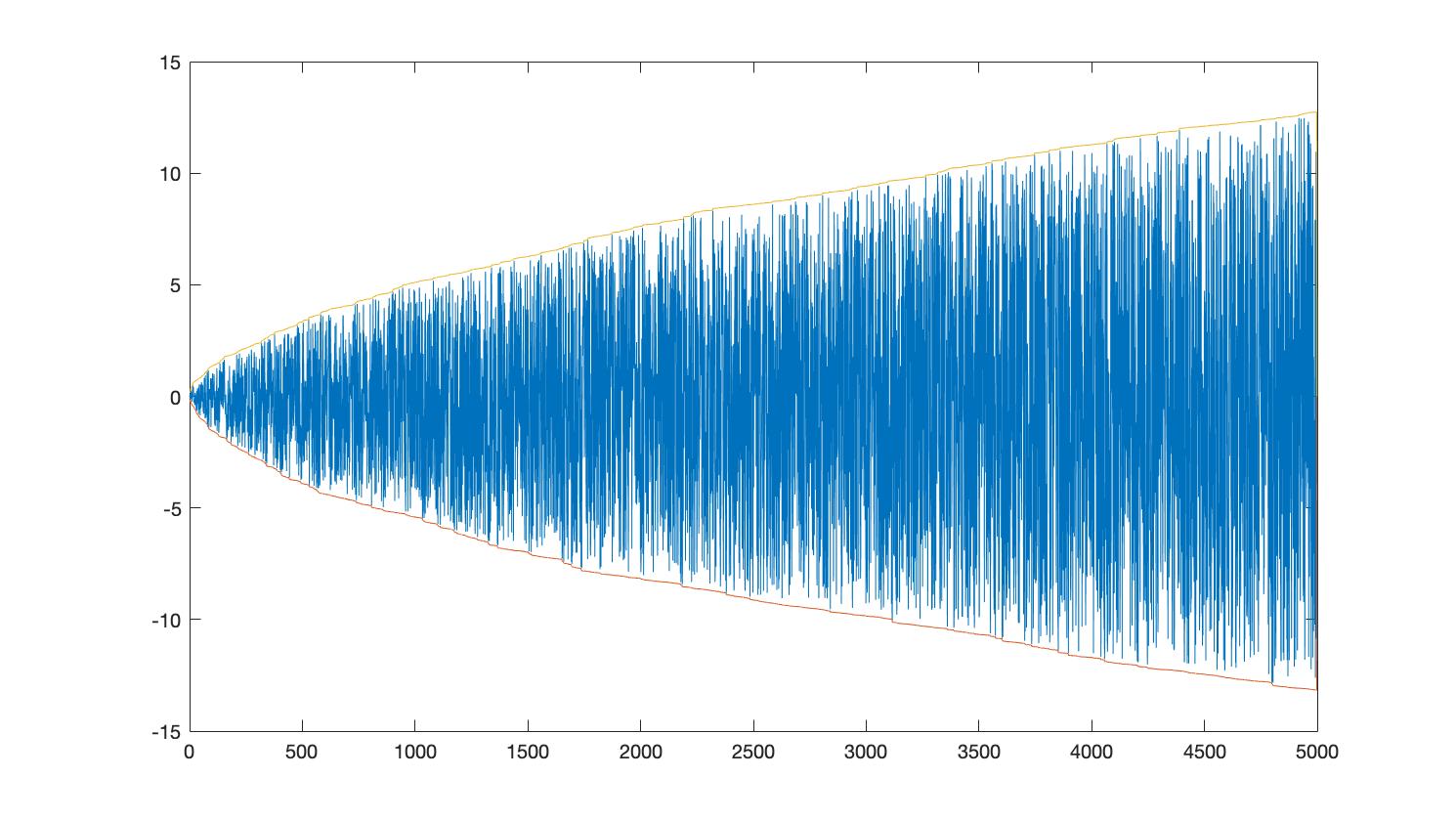}
\includegraphics[width=.48\textwidth]{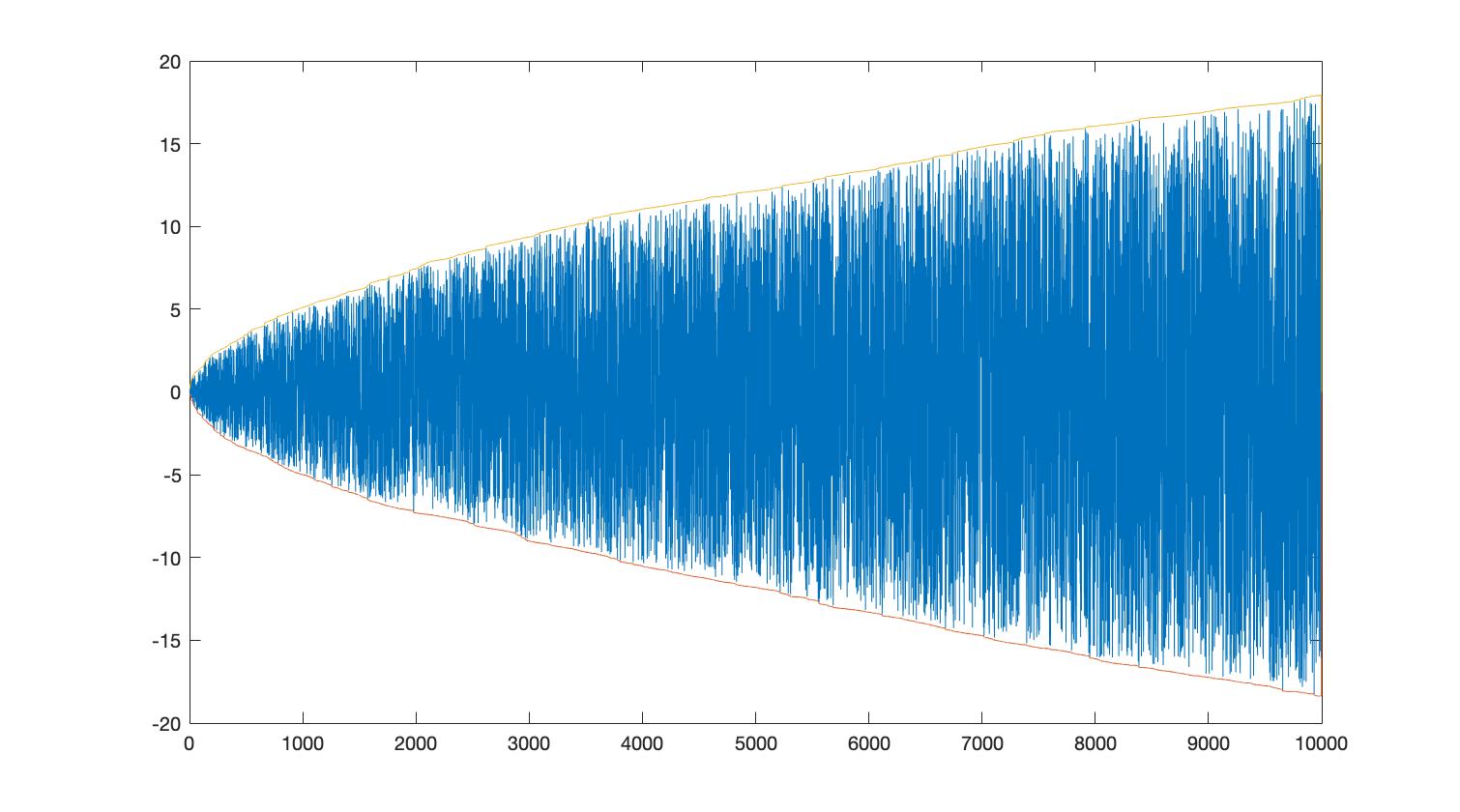}
\end{center}
\caption{Simulations of the allowed interval with $N = 5000$ and $10000$ (second and third simulations in Figure \ref{F:clusters}.}
\end{figure}

\subsection{Aggregation outside the unit disc}
\label{S:disc}

We have so far focused for convenience and ease on the situation in the upper-half plane, but it is in fact more standard to consider HL$(0)$ as an aggregation model outside the unit disc (the recent paper \cite{berger2020growth} by Berger, Procaccia and Turner being the exception). The model is defined similarly by replacing the upper-half plane $\H$ with the unit disc $\D$, and the elementary maps describing the slits are obtained from $F$ by conjugating with respect to a conformal map from $\C \setminus \D$ to $\H$. See e.g. \cite{norris2012hastings} for details. 

In order to keep the paper short and simple, we decided not to include results related to this case, \blu{although it should not be hard to see that our techniques apply to this setting as well}, at least as long as the two arms of the cluster remain macroscopically far apart (i.e., they do not shield the unit disc almost entirely). \blu{In particular, we believe that the following result holds}: for any $0< \alpha<1$, let $\tau_\alpha$ denote the first $k \ge 1$ such that the $k$th particle attaches at a position whose argument is in $[\pi(1 - \alpha), \pi(1 + \alpha)]$. Then, for any $0<\alpha<1$, as $n \to \infty$,
%
%
\begin{equation}
{\tau_\alpha} \sim \blu{\frac{2\pi^2 n (1-\alpha)^2}{\log n}}
\end{equation}
in probability, where $X_n \sim Y_n$ in probability as $n \to \infty$ means $X_n / Y_n \to 1$ in probability. Observe that the surprising phenomenon here is that $\tau_\alpha = o(n)$.

\subsection{Heuristic relation to instability in unconstrained Hastings--Levitov}

Compared to the unconstrained model (especially the disc one discussed in Section \ref{S:disc}), the only difference with our constrained model is that particles can only attach to the tree generated by the first particle. In other words, the constrained model describes the growth of a \emph{single} tree. From this perspective, the explosive growth demonstrated in Theorem \ref{T:main} shows a fundamental instability property which seems not to have been noticed before: left to its own devices, a single tree or fjord grows explosively. The existence of a shape theorem for regular (i.e., unconstrained) HL$(0)$ appears from this point of view nothing short of miraculous: essentially, this suggests that multiple fjords and subtrees are ``canceling each other out'' from the point of view of conformal maps.

We believe that this effect becomes less pronounced when particle sizes are rescaled according to the derivative of the conformal map at the location where particles attach -- this is for instance the situation of the HL$(\alpha)$, for $\alpha>0$ (see e.g. \cite{viklund2013small}, \cite{norris2021stability} and \cite{norris2022scaling} for a definition and some fine properties of the model). 

It has been proposed (and indeed proved for $\alpha<1$ in \cite{norris2021stability} for the unconstrained model) 
that compensating the size of the particles by a power of the derivative of the conformal map has an overall identical effect as keeping the size fixed but reweighing the law of the attachment by a power of the harmonic measure. In the constrained model however we expect that these two models behave differently for every $\alpha>0$, and the constrained model might well remain explosive even for values of $\alpha$ above $\alpha =1$.

\subsection{Sketch of proof of Theorem \ref{T:main} and heuristics}

The main argument for proving Theorem \ref{T:main} will consist in the following estimate. Call $\alpha(t) = R_{nt} / \sqrt{\log n}$. We wish to show that $\alpha (t) \asymp \sqrt{t}$. Roughly speaking we show that
\begin{equation}\label{ode}
\frac{d\alpha}{dt} \approx \frac1{8\alpha(t)}.
\end{equation}
This comes from a dyadic decomposition where we consider the contribution to the position of the front $R_{nt}$ of particles arriving at time $k \in [nt, n(t+h)]$ in a dyadic interval $I_j = [R_k - 2^{j+1}, R_k-2^j]$ of length $2^j$ at distance $2^j$ from the right front $R_k$. Here one should think that $2^j$ is much smaller than $n$, but at least $1/\sqrt{n}$; with the minimal distance $1/\sqrt{n}$ determined by the behaviour of the conformal map which uniformises a single particle. We show that the contribution to the push of the front of scale $j$ during this interval is, roughly,
$$
\Delta_j \approx nh \times \frac{|I_j|}{2\alpha \sqrt{\log n}} \times \frac{\log 2}{n 2^{j+1}} \approx \frac{h \log 2}{4 \alpha \sqrt{\log n}}
$$
where the first term $nh$ is the total number of particles arriving during the interval, the second term corresponds to the probability for a particle to fall in $I_j$, and the last term gives the average push caused by such a particle. Since the right hand side does not depend on $j$, we see that all scales contribute equally. This is the reason beyond the explosive growth.

Summed over all scales, this gives
$$
R_{(t+h)n } - R_{tn} \approx \sum_j \Delta_j \approx \frac{h \sqrt{\log n}}{8 \alpha (t)}
$$
where the constant $8$ comes from the fact the smallest dyadic scale starts at distance $1/\sqrt{n}$ from the tip of the interval. Dividing by $\sqrt{\log n}$ to express the left hand side in terms of $\alpha$, this suggests that $\alpha(t)$ satisfies the differential equation \eqref{ode}.

Equivalently, starting from very small $\alpha(0) $, we show that the time (measured in units of $1/n$) for $\alpha$ to double is roughly $\alpha^2$ (that is, we must add roughly $\alpha^2 n$ particles). More precisely, starting from an initial value $\alpha (0) = \alpha_n$, the solution to this ODE is
$$
\alpha(t) = (1/2)\sqrt{t + \alpha_n^2/4} ,
$$
so that, for arbitrary $\d >0$, the time it takes to reach $\alpha_n(1+ \delta)$ is
$$
t = 4\alpha_n^2 ((1+\delta)^2 -1) =  4\alpha_n^2 (2\delta + \delta^2) \approx 8\delta \alpha_n^2.
$$


\section{Proof of Theorem \ref{T:main}}
Fix $t \in (0,T]$ and $ \eps >0$ as in the statement of Theorem \ref{T:main}. Let $\delta>0$ denote a constant to be chosen later, possibly depending on $\eps$. Recall that for $1 \le k \le nT$ we denote by $\lambda_k $ the length of the interval of the cluster under the mapping out function $\Gamma_k$, so that $\l_k = L_k + R_k$.
Define
	\[ l (n) = \frac{(\log n)^3}{\sqrt{n}} , \]
and introduce the stopping times
	\[ \begin{split}
	T_1 & = \inf \{ i\geq 1 : \l_i \geq l (n) \} , \\
	T_k & = \inf\{ i\geq T_{k-1} : \lambda_{i} \geq (1+\d ) \l_{T_{k-1}} \} , \qquad k\geq 2.
	\end{split} \]
We start by showing that, once the length of the interval exceeds $ l (n)$, the time it takes for it to grow by a factor $1+\d $ cannot be too small.

\blu{Write $\mathcal{F}_n = \sigma ( x_1 , x_2 \ldots x_n )$ for the $\sigma$--algebra generated by the attachment locations up to the $n^{th}$ arrival, for $n\geq 1$.}
\begin{prop}[Lower bound on doubling time]
\label{P:doublingtimelower}
For any $k\geq 2$ it holds \blu{that} 
	\[ \P \left( T_k - T_{k-1} \geq \left( \frac{2\d}{1+\eps}\right)
	 \frac{n}{\log n}  \l_{T_{k-1}}^2 \Big| \mathcal{F}_{T_{k-1}} \right)
	\geq 1 - e^{ -c (\log n )^2} \]
for $n$ large enough. Here $c=c(\eps , \d )$ is a constant which may depend on $\eps $ and $\d $ but not on $n$ nor $k$.
\end{prop}
\begin{proof}
Write $m_k = \big(\frac{2\d}{1+\eps}\big) \frac{n}{\log n}  \l_{T_{k-1}}^2$ for brevity.
Then
	\[\begin{split}
	 \P  \left( T_k - T_{k-1} < m_k  \Big| \mathcal{F}_{T_{k-1}} \right)
	 & \leq \P ( \l_{T_{k-1} + m_k } \geq (1+\d ) \l_{T_{k-1}} )
	\\ &
	 \leq \P \left( R_{T_{k-1} + m_k } \geq  R_{T_{k-1}} + \frac{\d}2 \l_{T_{k-1}} \big| \mathcal{F}_{T_{k-1}} \right)
	 \\ & \, +
	 \P \left( L_{T_{k-1} + m_k } \geq  L_{T_{k-1}}  + \frac{\d}2  \l_{T_{k-1}} \big| \mathcal{F}_{T_{k-1}} \right) .
	\end{split}\]
We bound the first term in the right hand side. The second term can be bounded using exactly  the same argument.

The idea is to decompose over dyadic scales $[2^j, 2^{j+1})$ corresponding to the distance between the right end of the interval and the position of the particle that attaches.
Define
	\[\begin{split}
	 j_{min} & = \inf\Big\{ j \in \mathbb{Z} : 2^j \geq \frac 1 {\sqrt{n}} \Big\} , \\
	 j_{max} & = \inf \big\{ j \in \mathbb{Z} : 2^j \geq \sqrt{T \log n } \big\} . 	 \end{split} \]
For $i\geq 1$ and $j_{min} \leq j \leq  j_{max}$, let
  \[ \begin{split}
  I_{i, min} & = (R_i - 2^{j_{min}} , R_i ] , \\
  I_{i,j} & = (R_i - 2^{j+1} , R_i - 2^j ] . \\
	\end{split}
	\]
 \blu{These} intervals actually depend on time and are random, but they have fixed width $|I_{i,j} |=2^j$, and a given particle arriving at time $i+1$ attaches in one and only interval $I_{i,j}$ for some $j $.

 Let us describe the effect that a new arrival $x_{i+1}$ at time $i+1$ has on the interval $[-L_i , R_i ]$.

  \begin{lemma}\label{L:push1particle}
Suppose a new particle attaches at time $i$ at position $x \in I_{i,j}$ with
$j \geq j_{min}$, so that $x$ is at least distance $1/\sqrt{n}$ from $ R_i$. Then we have deterministically:
   $$
   R_{i+1} - R_i \le \frac1{2n (R_i-x)}
   $$
   and
   $$
   R_{i+1} - R_i \ge \frac1{2n (R_i - x)} - \frac{C}{(R_i-x)^3 n^2} ,
   $$
 for $C$ absolute constant.
     \end{lemma}
  \begin{proof}
  Observe that $R_{i+1}$ is obtained from $R_i$ by applying the map
  $$
  w \mapsto x + \sqrt{(w-x)^2 +1/n}
  $$
  so that
  $$
  R_{i+1} =x +  \sqrt{z^2 +1/n}, \text{ with } z = R_i -  x.
  $$
By convexity of the square root function,
\begin{align*}
R_{i+1} &\le x + \sqrt{z^2+ 1/n} \\
& \le x + z ( 1+ \frac{1}{2z^2n})  = x + z + \frac{1}{ 2zn}\\
& = R_i + \frac{1}{ 2zn} = R_i +\frac1{n(R_i-x)},
\end{align*}
which gives the desired upper bound.
On the other hand, for the lower bound, by  a Taylor expansion of $\sqrt{1+x}$ near zero,
$$
R_{i+1} \ge R_i + 1/(2nz) - C /(n^2 z^3)
$$
(which is valid since $z \ge 1/\sqrt{n}$ by assumption on $x$).
Since $z = R_i - x$ we get the desired inequality.
\end{proof}

\begin{lemma}\label{L:exp_push}
  In the same setting as above the following hold:
  \[ \begin{split}
   \E & ( R_{i+1} - R_i | x_{i+1} \in I_{i,j})\le \frac{\log 2}{n  2^{j+1}} , \\
   \E & ( R_{i+1} - R_i | x_{i+1} \in I_{i,j})\ge \frac{\log 2}{n  2^{j+1}} - \frac{C}{n^2 2^{3j}} ,
 \\  \var & (R_{i+1} - R_i | x_{i+1} \in I_{i,j}) \le \frac{C}{n^2 2^{2j}} , 
 \end{split} \]
\blu{where here and throughout the article  $C$ is an absolute constant that may change from line to line.}
\end{lemma}

\begin{proof}
  For the first upper bound, we use Lemma \ref{L:push1particle} and note that, conditionally on $x_{i+1} \in I_{i,j}$, $x_{i+1}$ is uniformly distributed over $I_{i,j}$, so that
  $$
  \E( R_{i+1} - R_i | x_{i+1} \in I_{i,j}) \le \int_{2^j}^{2^{j+1}} \frac{dx}{2nx2^j} = \frac{\log 2}{n  2^{j+1}},
  $$
  as desired.
  For the lower bound, we proceed similarly and note that
  \begin{align*}
    \E( R_{i+1} - R_i | x_{i+1}\in I_{i,j}) & \ge \int^{R_i- 2^j}_{R_i-2^{j+1}} \frac1{2(R_i - x)} \frac{dx}{n2^j} -  \int^{R_i- 2^j}_{R_i-2^{j+1}} \frac{C}{(R_i - x)^3} \frac{dx}{n^22^j}  \\
    &\ge \frac{\log 2}{n2^{j+1}}  - \frac{C}{n^2 2^j}\int_{2^j  }^{2^{j+1} } \frac{du}{u^3}
    \ge \frac{\log 2}{n2^{j+1}}  - \frac{C}{n^2 2^{3j}}
      \end{align*}
      as desired.
    Finally,
    \begin{align*}
  \var (R_{i+1} - R_i | x_{i+1} \in I_{i,j}) & \le \E ((R_{i+1} - R_i)^2 | x_{i+1} \in I_{i,j})\\
  & \le \int_{2^j}^{2^{j+1}} \frac{1}{n^2 x^2} \frac{dx}{2^{j+1}}\le \frac{C}{n^2 2^{2j}},
\end{align*}
as desired.
\end{proof}

Let us now go back to the proof of Proposition \ref{P:doublingtimelower}.
For $i \in [ T_{k-1}+1 , T_{k-1}+m_k ]$ let $x_i $ denote the \blu{attachment locations of the particles}, and for $j_{min} \leq j \leq j_{max} $, let
$$
\Delta_j = \sum_{i=T_{k-1}}^{T_{k-1}+m_k-1} 1_{\{x_{i+1} \in I_{i,j}\}} (R_{i+1} - R_i)
$$
denote the push to the right front caused by particles arriving in $I_{i,j}$ during the phase $[T_{k-1},T_{k-1}+m_k]$ (recall that the actual positions of the intervals $I_{i,j}$ change at each step, but their length is fixed).
Note that
	\[ \begin{split}
	\E ( \Delta_j | \mathcal{F}_{T_{k-1}} ) & =
	\sum_{i=T_{k-1}}^{T_{k-1}+m_k-1}
	\E ( 1_{\{x_{i+1} \in I_{i,j}\}} (R_{i+1} - R_i)| \mathcal{F}_{T_{k-1}} )
	\\ & = \sum_{i=T_{k-1}}^{T_{k-1}+m_k-1}
	\E ( R_{i+1} - R_i | x_{i+1} \in I_{i,j} ) \P ( x_{i+1} \in I_{i,j} )
	\\ & \leq \sum_{i=T_{k-1}}^{T_{k-1}+m_k-1}
	\frac{ \log 2}{n 2^{j+1} } \frac{|I_{i,j} | } {\l_{T_{k-1}}}
	= \frac{\log 2}{2n \l_{T_{k-1}}} m_k
	= \left( \frac {\d} {1+\eps} \right) \frac{\l_{T_{k-1}}}{\log_2 n} .
	\end{split} \]
\begin{lemma}\label{L:Deltaj}
For $k \geq 2$, $\d >0$ and any $\eps \in (0,1]$ it holds \blu{that} 
\[
\P \left( \sum_{j=j_{\min}}^{j_{\max}} \Delta_j
<  \frac{\d}2 \Big( 1-\frac{\eps}{8} \Big)  \l_{T_{k-1}}  \Big| \mathcal{F}_{T_{k-1}} \right)  \geq 1- e^{-c (\log n )^2 }
\]
for $n$ large enough ($n\geq n_0(\eps , T)$), where $c=c(\d , \eps )$ is a constant that may depend on $\d $ and $\eps$.
\end{lemma}
\begin{proof}
For $j_{\min} \leq j \leq j_{\max}$ set
\begin{align*}
D_j & = \Delta_j -\left( \frac {\d} {1+\eps} \right) \frac{\l_{T_{k-1}}}{\log_2 n} .
\end{align*}
In words, $D_j$ is obtained by subtracting from $\Delta_j$ an upper-bound on the conditional expectation of the push for each particle using Lemma \ref{L:exp_push}. Therefore, for each $j_{\min}\le j \le j_{\max}$, $D_j$ is the terminal value at time $m_k$ of a supermartingale. Furthermore, its total quadratic variation $Q_j$ at time $m_k$ (i.e.\ the sum of the conditional variances of the increments) is bounded by
  \[ \begin{split}
  Q_j &\le \sum_{i=T_{k-1}}^{T_{k-1}+m_k-1}
  \E ((R_{i+1} - R_i )^2 | x_{i+1} \in I_{i,j} ) \P (x_{i+1} \in I_{i,j} | \mathcal{F}_i)
  \\ & \leq
  \sum_{i=T_{k-1}}^{T_{k-1}+m_k-1}
  \frac{C}{n^2 2^{2j}} \frac{|I_{i,j} |}{\l_{T_{k-1}}}
  =  \frac{C}{n^2 2^{2j}} \frac{2^j}{\l_{T_{k-1}}} m_k \\
  & = C \left( \frac {2\d}{1+\eps} \right) \frac {\l_{T_{k-1}}}{2^j n \log n } =: b,
  \end{split}\]
 where the last equality defines $b$.
  Note also that each jump is bounded by $K = \frac1{n2^j}$. Applying Freedman's martingale inequality (Proposition 2.1 in \cite{Freedman}) we deduce that, with
  $$a = \frac{\delta \lambda_{T_{k-1}} }{ \log_2 n}  \left( \frac{1}{1+\eps /2} - \frac{1}{1+\eps} \right) $$ (note that $Ka \asymp b$ as $n \to\infty$),
  \begin{align*}
    \P \left( \Delta_j \geq \frac {\d \l_{T_{k-1}}}{\log_2 n} \frac{1}{(1+\eps /2)}\Big| \mathcal{F}_{T_{k-1}} \right)
    & = \P( D_j \ge a, Q_j \le b \, | \mathcal{F}_{T_{k-1}} )\\
    & \le \exp \left( - \frac{a^2}{Ka + b} \right)
    \le \exp \left( - c \frac{ n 2^j \l_{T_{k-1}} }{\log n } \right)
     \le e^{ - c (\log n )^2}
  \end{align*}
  since $2^j \geq 1/\sqrt{n}$ and $\l_{T_{k-1}} \geq \l_{T_1}\geq l(n) $. Here $c$ is a constant that may depend on $\d $ and $\eps$.
This bound  can be summed over all scales, to get that
  	\[ \begin{split}
  	\P \Big( \exists j\in [j_{min} , j_{max} ] : & \Delta_j \geq \frac {\delta \l_{T_{k-1}}}{(\log_2 n) (1+\eps /2)}  \Big| \mathcal{F}_{T_{k-1}} \Big)
  	\\ &  \leq ( j_{max} - j_{min} ) \P \Big( \Delta_j \geq \frac {\delta \l_{T_{k-1}}}{(\log_2 n)(1+\eps /2)} \Big| \mathcal{F}_{T_{k-1}} \Big)
  	\\ & \leq (\log_2 n) e^{ - c (\log n )^2}
  	 \leq e^{ - c' (\log n )^2}
  	 \end{split} \]
  for $n$ large enough. We used that $ j_{\max} - j_{\min} \leq   \frac 12 (\log_2 n )(1+\eps /4 )$ for $n$ large enough ($n\geq n_0(\eps , T)$).
It follows that, except on an event of probability at most $ \exp(- c (\log n)^2)$, it holds \blu{that} 
  \begin{align}
  \sum_{j=j_{\min}}^{j_{\max}} \Delta_j
  &\le \sum_{j=j_{\min}}^{j_{\max}}  \frac{\delta \l_{T_{k-1}} }{(\log_2 n)(1+\eps /2)}
  \leq  \frac{\d}2 \Big( 1-\frac{\eps}{8} \Big)  \l_{T_{k-1}} \label{pushscalej}
  \end{align}
for $n$ large enough, as long as $\eps \leq 1$, which we assumed.
\end{proof}

It remains to bound the displacement caused by particles arriving closer than $1/\sqrt{n}$ from the right front, which is encoded in
	\[ \Delta_{min} =
	 \sum_{i=T_{k-1}}^{T_{k-1}+m_k-1} 1_{\{x_{i+1} \in I_{i,min}\}} (R_{i+1} - R_i) .
	\]
\begin{lemma}\label{L:Deltamin}
For $k \geq 2$ and any $\eps , \d >0$ it holds \blu{that} 
\[
\P \left(
\Delta_{\min} < \left( \frac {6\d }{1+\eps} \right) \frac{\l_{T_{k-1}}}{\log n}
\Big| \mathcal{F}_{T_{k-1}} \right) \geq 1- e^{-c (\log n )^2 }
\]
for $n$ large enough, where $c$ is a constant that may depend on $\d $ and $\eps$.
\end{lemma}
\begin{proof}
This  can be bounded crudely: indeed, note that if $x_{i+1} \in I_{i,\min}$ then
  $$
  R_{i+1} - R_i \le 1/\sqrt{n}
  $$
  so
  $$
  \Delta_{\min} \le \frac{1}{\sqrt{n}}N_{\min}, \text{ where } N_{\min} = \#\{ T_{k-1}+1 \le i \le T_{k-1}+m_k : x_i \in I_{i,\min}\}.
  $$
  Note that $N_{\min}$ is dominated by a sum of independent indicator random variables where the probability of success is at most
  $$
  p_{\min} = \frac {2}{\sqrt{n} \l_{T_{k-1}}},
  $$
 since $\l_i \geq \l_{T_{k-1}}$ for all $i\geq T_{k-1}$ and $2^{j_{min} -1} < 1/\sqrt{n}$.
  Hence \begin{equation}
  \E(N_{\min}) \le m_k  p_{\min} = \left( \frac{4\d}{1+\eps} \right) \frac{ \sqrt{n}}{\log n} \l_{T_{k-1}} ,
  \label{E:ENmin}
  \end{equation}
  We conclude the proof of the lemma with the elementary fact on binomial random variables (which can also be deduced from Friedman's inequality) that for some $c>0$,
  $$
  \P( N_{\min} \ge (3/2) m_k p_{\min} ) \le \exp ( - c m_k p_{\min}) ,
  $$
 which implies that
 	\[ \begin{split}
 	\P \left( \Delta_{min} \ge \frac 32 \frac{m_k p_{min}}{\sqrt{n}} \Big| \mathcal{F}_{T_{k-1}} \right)
 	 & \leq \P ( N_{min} \geq (3/2) m_k p_{\min} )
 	\\ &  \leq \exp ( - c m_k p_{\min})
 	 \\ & \leq \exp ( -c (\log n )^2 )
 	 \end{split}
\]
as $\l_{T_{k-1}} \geq l(n) = (\log n)^3 / \sqrt{n} $ since $k \geq 2$ by assumption.
This, together with \eqref{E:ENmin}, concludes the proof.
  \end{proof}
To conclude the proof of Proposition \ref{P:doublingtimelower} we note that
 	\[ R_{T_{k-1} + m_k } - R_{T_{k-1}}
 	= \sum_{i=T_{k-1}}^{T_{k-1} + m_k -1 } ( R_{i+1} - R_i )
 	= \sum_{j=j_{min}}^{j_{max}} \Delta_j + \Delta_{min} . \]
from which, putting Lemmas \ref{L:Deltaj} and \ref{L:Deltamin} together, for $k\geq 2$ we obtain
	\[ \begin{split}
	\P \Big( R_{T_{k-1} + m_k } & - R_{T_{k-1}}  \geq \frac{\d}2 \l_{T_{k-1}} \Big| \mathcal{F}_{T_{k-1}} \Big)
	 \leq
	\P \bigg(
	 \sum_{j=j_{\min}}^{j_{\max}} \Delta_j
	 + \Delta_{\min} \geq   \frac{\delta}2 \l_{T_{k-1}}
	 \Big| \mathcal{F}_{T_{k-1}} \bigg)
	 \\ & \leq
	 \P \bigg(
	 \sum_{j=j_{\min}}^{j_{\max}} \Delta_j
	  \geq   \frac{\delta}2 \Big( 1-\frac{\eps }{8} \Big)\l_{T_{k-1}}
	 \Big| \mathcal{F}_{T_{k-1}} \bigg)
	 +
	 \P \bigg(  \Delta_{\min} \geq   \frac{\delta \eps }{16} \l_{T_{k-1}}
	 \Big| \mathcal{F}_{T_{k-1}} \bigg)
	 \\ &  \leq 2e^{-c (\log n )^2 }
	 \end{split}
	\]
for $n$ large enough ($n\geq n_0(\eps , T)$).
Similarly one shows that
	\[ \P \bigg( L_{T_{k-1} + m_k } - L_{T_{k-1}} \geq \frac{\d}2 \l_{T_{k-1}} \Big| \mathcal{F}_{T_{k-1}} \bigg)
	\leq 2e^{-c (\log n )^2 } , \]
and so
	\[ \P (T_k - T_{k-1} < m_k | \mathcal{F}_{T_{k-1}} )
	\leq 4 e^{-c (\log n )^2 } \]
for $n$ large enough.
\end{proof}	
To derive the upper bound on the growth, for any $t \in (0, T]$ define
	\[ S(t) = \inf\{ i\geq 1 : \l_i \geq \sqrt{t \log n } \} . \]
Note that $S(t)$ also depends on $n$, but we choose not to highlight this in the notation.
We will use Proposition \ref{P:doublingtimelower} to show that for any $\eps >0$
	\begin{equation}\label{Slower}
	 \P \left( \inf_{t\in (0,T]} \frac{S(t)}{nt} < \frac{1}{(1+\eps )^5  } \right)
	 \leq e^{-c(\log n)^2 }
	 \end{equation}
for $n$ large enough ($n \geq n_0(\eps , T )$). To see this, note that for any $k \geq 2$
	\[ \l_{T_k } \leq (1+\d ) \l_{T_{k-1}} + \frac 2 {\sqrt{n}}
	\leq (1+\d )^{k-1}\l_{T_1}  + \frac {2(k-1)}{\sqrt{n}} , \]
since the length of the interval increases by at most $1/\sqrt{n} $ in each direction upon a new arrival. Now using that $\l_{T_1} \leq l(n) + 2/\sqrt{n}$ we get
	\[
	\l_{T_k} \leq (1+\d )^{k-1} \left( l(n) + \frac 2 {\sqrt{n}} \right) + \frac {2(k-1)}{\sqrt{n}}
	\leq (1+\d )^{k-1} \frac {(\log n)^3}{\sqrt{n}} (1+\eps ) ,
	\]
where the last inequality holds as long as
	\begin{equation} \label{krequirement}
	2k \leq \eps (\log n)^3 .
	\end{equation}
It follows that if we let $K(t)$ denote the largest integer such that
	\begin{equation}\label{Kupper}
	 (1+\d )^{K(t)-1} < \frac{\sqrt{nt}}{(\log n)^{5/2} (1+\eps ) } ,
	 \end{equation}
then $\l_{T_{K(t)}} < \sqrt{t \log n } $, which forces $S(t) \geq T_{K(t)}$. Note that such $K(t)$ satisfies \eqref{krequirement} for $n$ large enough ($n\geq n_0(\eps , \d , T)$).
Thus
	\[ \begin{split}
	\P  \left( \inf_{t\in (0,T]} \frac{S(t)}{nt} < \frac{1}{(1+\eps )^5  } \right)
	& \leq \P  \left( \inf_{t\in (0,T] } \frac{T_{K(t)}}{nt} < \frac{1}{(1+\eps )^5  } \right)
	\\ & \leq \P  \left( \inf_{t\in (0,T] } \frac 1{nt} \sum_{k=2}^{K(t)} (T_k - T_{k-1} )  < \frac{1}{(1+\eps )^5 } \right) .
	\end{split} \]
Now, using that $ \l_{T_k} \geq (1+\d )^{k-1} l(n) $ for all $k\geq 2$ by definition of the stopping times $T_k$'s, it is easy to check that
	\[ \blu{\frac{1}{(1+\eps )^3 (1+\d ) }} \leq \frac 1{nt}\sum_{k=2}^{K(t)} m_k   \]
for all $t\in (0,T]$.
Taking $\d = \eps $ we get
	\begin{equation*}
	 \begin{split}
	\P  \left( \inf_{t\in (0,T]} \frac{S(t)}{nt}  < \frac{1}{(1+\eps )^5  } \right)
	&\leq \P \left( \exists t \in (0,T] : \,
	\sum_{k=2}^{K(t)} (T_k - T_{k-1} )  < \sum_{k=2}^{K(t)} m_k \right)
	\\ & = \P \left( \exists \, 2\leq K \leq K(T) : \,
	\sum_{k=2}^{K} (T_k - T_{k-1} )  < \sum_{k=2}^{K} m_k \right)
	\\ & \leq \sum_{K=2}^{K(T)} \P \left( \exists \, 2\leq k \leq {K} : T_k - T_{k-1} < m_k \right)
	\\ & \leq \sum_{K=2}^{K(T)} \sum_{k=2}^K \E ( \P (T_k - T_{k-1} < m_k | \mathcal{F}_{T_{k-1}} ) )
	\\ & \leq (K(T))^2 e^{-c (\log n)^2} \leq e^{-c' (\log n )^2 }
	\end{split}
	\end{equation*}
for $n$ large enough (depending on $\eps$ and $T$).  Here we have used Proposition \ref{P:doublingtimelower} in the fourth inequality, and \eqref{Kupper} in the last one.


\bigskip

We now turn to showing that the doubling time cannot be too large.

To start with, we need to control the time $T_1$ it takes to get to minimal length $l(n)= (\log n)^3 / \sqrt{n}$. It is easy to check that, while $\l_i \leq l(n)$, one has
	\[ R_{i+1} - R_i \geq \frac{1}{\sqrt{n} (\log n)^4} \]
for $n$ large enough ($n$ larger than an absolute constant): simply use\footnote{\blu{Here we are using the half--plane geometry.}} that the minimal push to the right is obtained when the new particle lands at the left end $x_{i+1} = -L_i$. It follows that $\l_{i+1} - \l_i \geq 1/(\sqrt{n} (\log n )^4)  $, from which we deduce that, deterministically, $T_1 \leq (\log n)^7$.

Now, once the interval has reached minimal length $l(n)$, it cannot grow too slowly, as shown in the next proposition.

\begin{prop}[Upper bound on doubling time]
\label{P:doublingtimeupper}
For any $k\geq 2$ it holds \blu{that} 
	\[ \P \left( T_k - T_{k-1} < 2\d(1+\eps )
	 \frac{n}{\log n}  \l_{T_{k-1}}^2 \Big| \mathcal{F}_{T_{k-1}} \right)
	\geq 1 - e^{ -c (\log n )^2} \]
for $n$ large enough ($n\geq n_0(\eps , T)$). Here $c=c(\d , \eps )$ is a constant which may depend on $\eps $ and $\d $ but not on $n$ nor $k$.
\end{prop}	
\begin{proof}
This can be proved reasoning as in Proposition \ref{P:doublingtimelower}.
Write
	\[ n_k = 2\d(1+\eps )
	 \frac{n}{\log n}  \l_{T_{k-1}}^2\]
for brevity, and define
	\[
	\tilde{\Delta}_j = \sum_{i=T_{k-1}}^{T_{k-1}+n_k-1} 1_{\{x_{i+1} \in I_{i,j}\}} (R_{i+1} - R_i)
	\]
for $\tilde{j}_{min} \leq j \leq j_{max}$, where $j_{max}$ is as before, and \blu{$\tilde{j}_{min}$} is chosen, depending on $\eps$, so that
	\[ \left( \frac{\log 2}{2n} - \frac{C}{n^2 2^{2j}} \right) \geq \frac{\log n}{2n} \frac{1+\eps /2}{1+\eps} , \]
where $C$ is the absolute constant in Lemma \ref{L:push1particle}.
This ensures that
	\[ \E ( \tilde{\Delta}_j | \mathcal{F}_{T_{k-1}} ) \geq
	\d ( 1+\eps /2) \frac{ \l_{T_{k-1}}}{\log_2 n} . \]
Thus
	\[  - \tilde{\Delta}_j + \d ( 1+\eps /2) \frac{ \l_{T_{k-1}}}{\log_2 n} \]
is the terminal value at time $n_k$ of a supermartingale, and one again concludes by Freedman's inequality that
	\[ \P \left( \tilde{\Delta}_j < \frac{ \delta \l_{T_{k-1}}}{\log _2 n} \Big| \mathcal{F}_{T_{k-1}} \right) \leq e^{-c (\log n)^2 } \]
for $n$ large enough, depending on $\eps $ and $T$, and a constant $c=c(\d , \eps )$.
This can be summed over all scales $\tilde{j}_{min} \leq j \leq j_{max} $ to get that, since
	 	\[ R_{T_{k-1} + n_k } - R_{T_{k-1}}
 	= \sum_{i=T_{k-1}}^{T_{k-1} + n_k -1 } ( R_{i+1} - R_i )
 	\geq  \sum_{j=\tilde{j}_{min}}^{j_{max}} \Delta_j  , \]
for $k\geq 2$ it holds \blu{that} 
	\[ \begin{split}
	\P \Big( R_{T_{k-1} + n_k } - R_{T_{k-1}} < \frac{\d}2 \l_{T_{k-1}} | \mathcal{F}_{T_{k-1}} \Big)
	& \leq
	\P \bigg(
	 \sum_{j=\tilde{j}_{\min}}^{j_{\max}} \tilde{\Delta}_j
	  <   \frac{\delta}2 \l_{T_{k-1}}
	 \Big| \mathcal{F}_{T_{k-1}} \bigg)
	 \\ &  \leq e^{-c (\log n )^2 }
	 \end{split}
	\]
for $n$ large enough. The results follows by observing that the same bound holds for $L_{T_{k-1} + n_k } - L_{T_{k-1} }$.
\end{proof}

We now discuss how the above result allows to conclude that the interval cannot grow too slowly, once it has reached minimal length $l(n)$. We show that, with $t_0(n) = (\log n)^8/n$,
	\begin{equation}\label{Supper}
	 \P  \left( \sup_{t\in [t_0(n),T]} \frac{S(t)}{nt} > (1+\eps )^5   \right)
	\leq e^{-c (\log n)^2}
	\end{equation}
for $n \geq n_0(\eps , T)$.
For $t \in [t_0(n) ,T] $ let $\tilde{K}(t)$ denote the smallest  integer such that
	\[ (1+\d )^{\tilde{K}(t)-1} > \frac{\sqrt{nt}  }{(\log n )^{5/2}} , \]
so that
	\begin{equation}\label{KtildaUP}
	 (1+\d )^{\tilde{K}(t)-2} \leq  \frac{\sqrt{nt}  }{(\log n )^{5/2}} .
	 \end{equation}
Then $\lambda_{T_{\tilde{K}(t)}} \geq \sqrt{t \log n } $, which forces $S(t) \leq T_{\tilde{K}(t)}$.
Moreover, since $\l_{T_k } \leq (1+\d ) \l_{T_{k-1}} + 2 /\sqrt{n}$
for all $k \geq 2$, it is easily checked that, by the choice of $\tilde{K}(t)$,
	\[  \frac{1}{nt} \sum_{k=2}^{\tilde{K}(t)} n_k \leq  (1+\eps )^2 (1+\delta )^2  \]
for all $t \in (0,T]$ and $n$ large enough, depending only on $\d  , \eps$.
Thus, taking again $\d = \eps$, we have that
	\[ \begin{split}
	\P \bigg( \sup_{t \in [t_0(n) , T] }  \frac{S(t)}{nt} & > (1+\eps )^5 \bigg)  \leq
	\P \left( \sup_{t \in [t_0(n) , T] } \frac{T_{\tilde{K}(t)}}{nt} > (1+\eps )^5 \right)
	\\ & \leq
	\P \left( \sup_{t \in [t_0(n) , T] }  \frac{T_1}{nt} > \eps  \right) +
	\P \left( \sup_{t \in [t_0(n) , T] }  \frac 1{nt}
	\sum_{k=2}^{\tilde{K}(t)} (T_k - T_{k-1} ) > (1+\eps )^4 \right)
	\\ & \leq \P \left(  T_1 > \eps (\log n)^8  \right) +
	\P \left( \exists\,  t \in [t_0(n) , T ] : \,
	\sum_{k=2}^{\tilde{K}(t)}  (T_k - T_{k-1} ) > \sum_{k=2}^{\tilde{K}(t)}n_k \right)
	\\ & \leq
	\P \left( \exists \, 2\leq K \leq \tilde{K}(T) : \, \sum_{k=2}^{K}  (T_k - T_{k-1} ) > \sum_{k=2}^{K}n_k \right)
	\\ & \leq \sum_{K=2}^{\tilde{K}(T)} \P \left( \exists \, 2\leq k \leq K : \, T_k - T_{k-1}  > n_k \right)
	\\ & \leq \sum_{K=2}^{\tilde{K}(T)} \sum_{k=2}^{K}  \E ( \P (T_k - T_{k-1} > n_k | \mathcal{F}_{T_{k-1}} ) )
	\\ & \leq (\tilde{K}(T))^2  e^{-c (\log n)^2} \leq e^{-c' (\log n )^2 }
	\end{split} \]
for $n$ large enough ($n\geq n_0(\eps , T)$).
The fourth inequality above holds for $n$ large enough so that $\eps (\log n)^8 \geq (\log n)^7$, which makes $\P (T_1 > \eps (\log n)^8 )$ vanish. For the last one we have used  \eqref{KtildaUP}.

Putting together the upper and lower bounds \eqref{Slower} and \eqref{Supper} for $S(t)$ we gather that for any $\eps \in (0,1] $ and $T>0$ it holds \blu{that} 
	\begin{equation}\label{Sconc}
	 \P \left(
	 \frac{nt}{(1+\eps )^5} \leq S(t) \leq nt (1+\eps )^5
	 \quad \forall t \in [t_0(n) , T ] \, \right)
	\geq 1-e^{-c(\log n)^2}
	\end{equation}
for $n$ large enough (depending on $\eps , T$). To deduce that a similar concentration result holds for the length $\l_{nt}$ we use that
	\[ \sqrt{t \log n } \leq S(t) \leq \sqrt{t \log n } + \frac 2{\sqrt{n}} \]
for all $t \in (0,T]$. This tells us that, on the event in \eqref{Sconc},
	\[ \l_{nt(1+\eps )^5 } \geq \l_{S(t)} \geq \sqrt{t \log n } \qquad \forall t \in [t_0(n) ,T] , \]
which implies
	\[ \l_{nt} \geq \sqrt{\frac{t \log n }{(1+\eps )^5}} \geq (1-3\eps ) \sqrt{t \log n }
	\qquad \forall t \in [t_0(n) (1+\eps )^5  ,T] \]
for $\eps$ small enough ($\eps \leq \eps_0$ with $\eps_0$ absolute constant).

For the upper bound note that, again on the event in \eqref{Sconc},
	\[ \l_{nt(1-\eps )^5} \leq \l_{S(t) } \leq \sqrt{t \log n } + \frac 2 {\sqrt{n}} , \]
which implies
	\[ \l_{nt} \leq \sqrt{\frac{t \log n }{(1-\eps )^5}} + \frac 2 {\sqrt{n}}
	\leq (1+3\eps ) \sqrt{t \log n } \qquad \forall t  \in [t_0(n) (1-\eps )^5 , T ] \]
for $\eps$ small enough ($\eps \leq \eps_0$ with $\eps_0$ absolute constant). This shows that for any $\eps >0$ small enough, and any $T>0$, there exists $n_0(\eps , T)$ such that
	\begin{equation} \label{lambda_concentration}
	 \P \left( (1-3\eps ) \sqrt{t \log n} \leq \l_{nt} \leq (1+3\eps ) \sqrt{t \log n } \quad
	\forall t \in \Big[ \frac{(\log n)^9}{n} , T \Big] \right)  \geq 1-e^{-c(\log n)^2}
	\end{equation}
for all $n\geq n_0(\eps , T)$,  which concludes the proof of  \eqref{Lmain}. \\

It remains to deduce that $R_{nt}$ and $L_{nt}$ are both concentrated around $\frac 12 \sqrt{t \log n }$. We show that there exists an absolute constant $\eps_0$ such that for all $\eps \in (0, \eps_0 )$ it holds
	\[ \P \left( \frac{\sqrt{t \log n }}{2(1+\eps )^6} \leq R_{nt} \leq \frac 12 \sqrt{t\log n } (1+\eps )^6 \quad
	\forall t \in \Big[ \frac{(\log n)^9}{n} , T \Big] \right) \geq 1 - e^{-c (\log n )^2} \]
for $n$ large enough, and $c$ constant independent of $n$. To this end, recall the definition of $K(t)$ from \eqref{Kupper}, and note that for any $t\in (t_0(n),T]$, restricting to  the high probability event $\{T_k - T_{k-1} \geq m_k \mbox{ for all } 2 \leq k \leq K(T) \}$, we have
	\[ T_{K(t)} \geq \sum_{k=2}^{K(t)} ( T_k - T_{k-1} ) \geq \sum_{k=2}^{K(t)} m_k \geq \frac {nt}{(1+\eps )^5 } . \]
It follows that
	\[ R_{\frac{nt}{(1+\eps )^5}} \leq R_{T_{K(t)}}
	= R_{T_1} + \sum_{k=2}^{K(t)} (R_{T_k} - R_{T_{k-1}} ) \]
for all $t \in (t_0(n),T ]$.
Restricting further to the event $\{ T_k - T_{k-1} \leq n_k  \mbox{ for all } 2\leq k \leq \max\{K(T) , \tilde{K}(T)\} \}$ we have that $R_{T_k} - R_{T_{k-1}} \leq R_{T_{k-1} + n_k } - R_{T_{k-1}} $. It can be shown, using exactly the same argument as  the one of Proposition \ref{P:doublingtimelower}, that
	\[ \P \left( R_{T_{k-1} + n_k } - R_{T_{k-1}} \leq \frac \d 2 ( 1+\eps )^2  \l_{T_{k-1} }
	 \mbox{ for all } 2\leq k \leq K(T) \right) \geq 1-e^{-c (\log n)^2 } \]
for $n$ large enough ($n\geq n_0(\eps , T)$). Thus, with probability exceeding  $1-3e^{-c (\log n)^2}$, we have that
	\[ R_{\frac{nt}{(1+\eps )^5}} \leq \l_{T_1} + \frac \d 2 (1+\eps )^2 \sum_{k=2}^{K(t)} \l_{T_{k-1}} 
	\blu{\leq \frac{(\log n )^3 + 2}{\sqrt{n}} + \frac{(1+\varepsilon)^2}{2} \sqrt{t \log n }}
	\leq \frac{(1+\eps )^3}2 \sqrt{t \log n }  \]
for all $t \in (t_0(n),T]$,
where in the second inequality we have used the definition of $K(t)$. Rearranging, this gives
	\[ \P \left( R_{nt} \leq \frac{(1+\eps )^6}2 \sqrt{t \log n } \quad \forall t \in (t_0(n),T] \right)
	\geq 1-2e^{-c(\log n)^2} .\]
Similarly one shows that
	\[ \P \left( L_{nt} \leq \frac{(1+\eps )^6}2 \sqrt{t \log n } \quad \forall t \in (t_0(n),T] \right)
	\geq 1-2e^{-c(\log n)^2} .\]
Combining these with the lower bound in \eqref{lambda_concentration} we also deduce that, with probability exceeding $1-5   e^{-c(\log n)^2}$,
	\[ R_{nt} = \l_{nt} - L_{nt} \geq \sqrt{t\log n } \left( \frac{1}{(1+\eps )^3} - \frac 12 (1+\eps )^6 \right)  \geq \frac 1{2 (1+\eps )^6} \sqrt{t\log n } \]
for all $t \in [(\log n)^9 /n , T ]$,
where the last inequality holds as long as $\eps \leq \eps_0$ for some absolute  constant $\eps_0 >0$. The same holds for $L_{nt}$, which concludes the proof of Theorem \ref{T:main}.

\section{Proof of Theorem \ref{T:geometric}}

We first prove that the height of the cluster is deterministically bounded. This comes from the simple observation that the half-plane capacity of $\cK_{nt}$ is deterministically given by $t$ and the following simple lemma: if $K$ is a compact $\H$-hull such that $x+ i y \in K$ for some $x \in \R, y >0$ then $$\hcap (K) \ge y^2/2.$$
This can be proved using a simple reflection argument about the line $\{\Re(z) = x\}$ and monotonicity of the half-plane capacity; we leave the details to the reader (this is implicit in \cite{lawler2008conformally}, and explicitly proved in the note \cite{Lalley_etal}, see also Proposition 6.1 of \cite{BN} for a closely related statement).

As a consequence, since $ \hcap (\cK_{nt}) = t$ then we deduce that
$$
\max_{z \in \cK_{nt}} \{ \Im(z)\} \le \sqrt{2t},
$$
as desired.

Now let us turn to the diameter. The argument is based on estimating the harmonic measure of $\cK_{nt}$ viewed from infinity (which is conformally invariant) and relating this to the diameter, given that the cluster is necessarily very flat by the above. We start with the lower bound on the diameter which is slightly easier.

Fix $t>0$. Let $L = (1 - \eps) \sqrt{t \log n}$ and denote by $K = \cK_{nt}$ for ease of notations. Let $z = i y$ with $y>0$ very large. Suppose for contradiction that $\diam (K) \le L$, so $K$ could fit in a rectangle $\cR$ of height $\sqrt{2t}$ and width $L$, in the upper half plane and with the bottom side on the real line, and which we can assume to be centered without loss of generality. See Figure \ref{F:composition}. Then applying the conformal map $g_K$, the hull $K$ is by definition mapped to the interval $[- L_t, R_t]$. Let $\tilde \cR = g_K (\cR \setminus K)$. If we then apply the conformal map $g_{\tilde \cR}$, then by monotonicity, we have
$$
R_t \le r, L_t \le \ell
$$
where $- \ell $ and $r$ respectively are
the extreme points of $g_{\cR} ( \partial \cR \cap \H)$. (See Figure \ref{F:composition}). In fact by symmetry note that $\ell = r$. Furthermore we claim that
\begin{equation}\label{E:goal_ell}
\ell \le L/2 + C,
\end{equation} where $C$ is a constant depending only on $t$. From here the contradiction (and thus the lower bound) follows, since we know that $\min(R_t, L_t)  \ge (1 - \eps/2)(1/2) \sqrt{t \log n}$ with probability converging to $1$ as $n \to \infty$.

\begin{figure}
  \begin{center}
    \includegraphics[width=.8\textwidth]{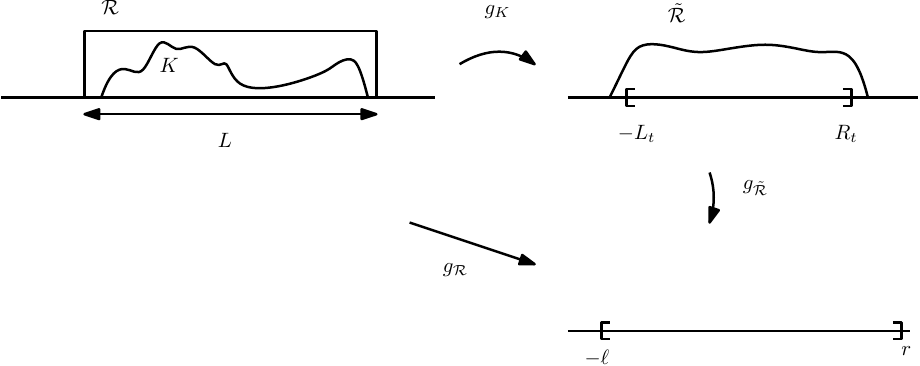}
  \end{center}
  \caption{Composition of conformal maps $g_{\cR} = g_{\tilde \cR}\circ g_K$.}
  \label{F:composition}
\end{figure}

This can be seen directly because the conformal map $g_{\cR}$ is known explicitly (it is a Schwarz--Christoffel map, since $D  = \H \setminus \cR$ is a polygon), but it can be seen more robustly as follows.
Fix $z = i y$ with $y>0$ very large. Let $\tau_D$ denote the first time at which a Brownian motion $B$ leaves the domain $D$.
Let $\tau$ denote the first time that the imaginary part of $B$ reaches $\sqrt{2t}$. Let $E$ be the event that $|\Re (B_\tau)| \le L/2 $.
\begin{align*}
\P^{iy} (B_{\tau_D} \in \cR)  & = \P^{iy} ( E) + \P^z(B_{\tau_D} \in \cR; E^c).
\end{align*}
Now observe that $\P^z(E)$ may be computed through the integral over the interval $[-L/2, L/2]$ of a Cauchy distribution with scale parameter $y'= y - \sqrt{2t}$: hence after a change of variables,
$$
\P^{iy} (E) = \int_{-L/(2y')}^{L/(2y')} \frac{1}{\pi(x^2 + 1)} dx \sim \frac{L}{\pi y}
$$
as $y \to \infty$ (and $t>0$ fixed), where $a_n \sim b_n$ means that the ratio $a_n / b_n \to 1$. Now let us turn to $E^c$, and observe that if \blu{$B_{\tau_D} \in \cR$} and $E^c$ both hold, then necessarily Brownian motion hits one of the vertical sides of the rectangle $\cR$ before the real line. In words, if $I_1, I_2$ are the vertical sides, and $D^i = \H \setminus I_i$, then
$$
(\{ B_{\tau_D} \in \cR\} \cap E^c )\subset \left\{B_{\tau_{D^1}} \in I_1 \right\} \cup \left\{B_{\tau_{D^2}} \in I_2 \right\}
$$
Hence by symmetry between $I_1$ and $I_2$,
$$
\P^{iy} (B_{\tau_D} \in \cR; E^c) \le 2 \P^{iy} ( B_{\tau_{D^1}} \in I_1 ) \sim 2 \blu{\frac{\text{cap} (I_1)}{\pi y}} , 
$$
\blu{where $\text{cap} ( I ) $ denotes the capacity of $I$ from infinity in $\H$ (see \cite{BN}, Section 4.1) which, we remark, differs from the half--plane capacity $\text{hcap}(I)$.} 
Altogether, since \blu{$\text{cap} (I_1) = 2\sqrt{2t} $ (see \cite{BN}, Section 6.1)}, we see that
\begin{equation}\label{E:L}
\lim_{y \to \infty }y \P^{iy} (B_{\tau_D} \in \cR) \le  L/\pi + 4\sqrt{2t}/\pi.
\end{equation}
To get the desired inequality on $\ell$ we simply argue by conformal invariance:
\begin{equation}\label{E:ell}
\P^{iy} ( B_{\tau_D}  \in\cR)  = \P^{g_{\cR} (iy)} ( B_{\tau_{\H}} \in [- \ell, \ell]) \sim \frac{2 \ell}{\pi y}
\end{equation}
by a similar integration of the Cauchy distribution and the fact that $g_{\cR} (iy) \sim i y $. Multiplying by $y$ and letting $y \to \infty$, and comparing \eqref{E:ell} with \eqref{E:L}, we get
$$
\frac{2\ell}{\pi} \le  \frac{L}{\pi} + 4\sqrt{2t}/\pi,
$$
which proves \eqref{E:goal_ell} with $C =  \sqrt{32t}$. This proves the desired lower bound on $\diam (\cK_{nt})$.

Now let us turn to the corresponding upper bound. We again argue by contradiction. Let $L = (1+ \eps) \sqrt{t \log n}$. Suppose $\diam (K) \ge L $. Then there exists $a,b \in K$ with $|a- b | \ge  L$. Since $K$ is connected there is a curve in $K$ connecting $a$ to $b$ (and so its maximal height is at most $\sqrt{2t}$). We argue similarly as above. Let $I_1, I_2$ be the vertical segments connecting $a$ and $b$ respectively to the real line; call $x_1$ and $x_2$ the respective base points on $\R$ of $I_1$ and $I_2$. By translation we can assume $x_1 = -x_2$ and by the triangle inequality $|x_1 - x_2 | \ge L - O(1)$. Let $E = \{ \tau_H \le \tau_{\H} \} = \{ B_{\tau_H} \in K\}$.
\begin{align*}
  \P^{iy}( B_{\tau_\H} \in [x_1, x_2] ) & = \P^{iy} ( B_{\tau_\H} \in [x_1, x_2] ; E ) + \P (B_{\tau_\H} \in [x_1, x_2]  ; E^c).
\end{align*}
For the same topological reasons as above, if $B_{\tau_\H} \in [x_1, x_2]  $ and $E^c$ both hold, it must be the case that the Brownian motion $B$ touched either $I_1$ or $I_2$ before the real line (indeed, to enter the interval $[x_1, x_2]$ the Brownian curve cannot avoid $I_1$ and $I_2$ without touching the curve connecting $a$ and $b$, but then $E$ holds). Therefore,
$$
 \P^{iy}( B_{\tau_\H} \in [x_1, x_2] )  \le \P^{iy}(E) + \P^{iy} ( B_{\tau_{D_1}} \in I_1) + \P^{iy} ( B_{\tau_{D_2}} \in I_2).
$$
Now we multiply by $y$ and let $y \to \infty$. Using the same argument as above, we see that the left hand side converges to $|x_1 - x_2| /\pi  = L /\pi - O(1). $ The right hand side has three terms. The first, by conformal invariance of Brownian motion, converges to $| L_t - R_t | / \pi$. The second and third respectively converge to $\text{cap} (I_1)$ and $\text{cap}(I_2)$, both of which are $O(1)$.
Consequently, we obtain
$$
L \le  | L_t - R_t | + O(1).
$$
By Theorem \ref{T:main}, the probability of the above tends to zero.

\paragraph{Acknowledgements.} This work took place while N.B. was visiting professor at the University of Rome La Sapienza, whose hospitality is gratefully acknowledged. 

We would like to thank James Norris for some useful and inspiring conversations, both in relation with this project and over the years. We are also grateful to the anonymous referee for a careful reading and thoughtful comments on the paper. 

NB’s work was supported by a University of
Vienna start-up grant, and FWF grant P33083 on “Scaling limits in random conformal geometry”.


\bibliography{HLbiblio}
\bibliographystyle{plain}

\end{document}